\documentclass[amscd,amssymb,verbatim,12pt]{amsart}
\usepackage{epsfig}
\usepackage{graphicx}
\newif\ifAMS
\IfFileExists{amssymb.sty}
  {\AMStrue\usepackage{amssymb}}
  {\usepackage{latexsym}}
  \usepackage{enumerate}
\theoremstyle{plain}

\newtheorem*{inductive}{Theorem \ref{inductive}}

\newtheorem{Thm}{Theorem}[section]

\newtheorem{Lem}[Thm]{Lemma}

\theoremstyle{definition}
\newtheorem{Def}{Definition}

\theoremstyle{remark}
\newtheorem{Rem}{Remark}

\DeclareMathOperator{\area}{area}
\DeclareMathOperator{\vol}{vol}

\newcommand{\interior}{^{ \kern-5pt ^\circ}}
\newcommand {\bd}{\partial}

\newcommand {\diam}{\text{diam}}

\begin{document}
\title
{Uryson width and volume }

\author
{Panos Papasoglu }



\email {papazoglou@maths.ox.ac.uk}

\address
{Mathematical Institute,  University of Oxford, Andrew Wiles Building, Woodstock Rd, Oxford OX2 6GG, U.K. }

\address
{ }

\begin{abstract} We give a short proof of a theorem of Guth relating volume of balls and Uryson width. The same
approach applies to Hausdorff content implying a recent result of Liokumovich-Lishak-Nabutovsky-Rotman.
We show also that for any $C>0$ there is a Riemannian metric $g$ on a  3-sphere such that $\vol (S^3,g)=1$
and for any map $f:S^3\to \mathbb R^2$ there is some $x\in \mathbb R^2$ for which $\diam (f^{-1}(x))>C$-answering a question of Guth.
\end{abstract}
\maketitle
\section{Introduction}

The Uryson width is a notion of topological dimension theory that was brought to the realm of Riemannian Geometry
by Gromov \cite{Gro-fill}, \cite{Gro-width},\cite{Gro-large}. It appears quite naturally in the context of thick-thin decompositions of Riemannian manifolds \cite{Gro-inter}.
Intuitively small $k$-Uryson width means that an $n$-dimensional space `collapses' to a $k$-dimensional space (where we assume $k<n$). For example if we consider a torus $T^2=S^1\times S^1$
where one of the $S^1$'s has very small length $\epsilon $ and the other has, say, length 1 then $T^2$ is `close' (collapses) to the circle of length 1-a lower dimensional manifold.
In case that our $n$-dimensional space is non compact, bounded $k$-Uryson width means that the space is `close' to a $k$-dimensional space.
 
We recall now the precise definition: if $X$ is a metric space we say that $X$ has $q$-Uryson width $\leq W$ if there exists a $q$-dimensional simplicial complex $Y$ and a continuous map
$\pi: X\to Y$ such that every fiber $\pi ^{-1}(y)$ has diameter $\leq W$. We write then that $UW_q(X)\leq W$.

Guth (\cite {Gu1},\cite{Gu2}) proved the following theorem answering a conjecture of Gromov:

\begin{Thm}\label{volume} There exists $\epsilon _n>0$ so that the following holds. If $(M^n,g)$ is a closed Riemannian manifold and there exists
a radius $R$ such that every ball of radius $R$ in $(M^n,g)$ has volume at most $\epsilon _n R^n$ then $UW_{n-1}(M^n,g)\leq R$.
\end{Thm}

Guth conjectured something stronger that applies to general metric spaces and uses Hausdorff content instead of volume. This was shown recently by Liokumovich-Lishak-Nabutovsky-Rotman \cite{LLNR}. 

The proofs of all these results are somewhat technical as they associate to the space some nice coverings and then approximate the space
by the rectangular nerve of these coverings (a method introduced by Gromov in \cite{Gro-coh} and applied in \cite{Gro-fill}, p.130 to the case
of manifolds with a lower Ricci curvature bound). They also use various generalizations of the isoperimetric inequality.

Our aim in this paper is to give a direct proof relying only on the co-area inequality.
Our method gives also a new weaker sufficient condition for a space to have small Uryson width. Guth in \cite{Gu2} discusses the relationship between classical topological dimension theory and the quantitative version of this theory for manifolds. 
In this spirit we show the following quantitative characterization of small Uryson width that resembles the inductive definition of classical topological dimension.
\begin{inductive}  There exist $\epsilon _n>0$ so that the following holds. Suppose $X$ is a proper metric space
and there exists
a radius $R$ such that every ball $B(x,R)$ in $X$ is contained in an open set $U$ with the properties:

1. $U\subseteq B(x,10R)$.

2. $HC_{n-1}(\bd U)\leq \epsilon _nR^{n-1}$.

Then $UW_{n-1}(X)\leq R$.

\end{inductive}
We denote above by $HC_n$ the $n$-dimensional Hausdorff content of a metric space-for the definition see section 2.

We note also that Nabutovsky \cite{Na} used the method of this paper to give better bounds for the dimensional constant in the celebrated systolic inequality of Gromov.

We give now the idea of the proof of theorem \ref{volume}: Let's say that we have a thickened plane $P$, so locally the volume growth is much smaller than $r^3$.
Then we cut $P$ in pieces of small diameter $<D$ by a `thickened grid' $G$ (we call this a $D$-{\it separating subset} in sec. 2). Using the co-area inequality (see lemma \ref{small-cut}) we show that there is a thickened grid that locally has volume growth much smaller than $r^2$ so by induction it admits a map $f$ to a 1-dimensional complex $\Sigma $ with small fibers. By adding finite cones to $\Sigma $ we may extend
$f$ to the pieces of $P\setminus G$, so to the whole of $P$. It is easy to see that the fibers of this map have small diameter.
We note that our approach is reminiscent of the minimal hypersurface method of Schoen-Yau \cite{SY1}, \cite{SY2} which was used also by Guth \cite{Gu3} in a context
similar to ours. Indeed our `thickened' grid has a rough `minimal area' property and in some cases can indeed be replaced by a smooth hypersurface. One novelty
of our approach is that even when we deal with manifolds we need to consider more general spaces like the grid in the above example.

In section 2 we carry out this proof in the more general context of compact metric spaces and in section 3 we deal with the non-compact case.

These results imply that there is some $C>0$ such that if $\vol (S^3,g)=1$ then there is a map $f:S^3\to \Sigma $ where $\Sigma $ is a 2-complex and the `fibers'
of the map have diameter bounded by $C$. Guth in \cite {Gu4} asks whether $\Sigma $ may be replaced by $\mathbb R^2$. In section 4 we give examples
showing that the answer is negative.

I am grateful to Stephane Sabourau for pointing out mistakes in an earlier version of this paper and making suggestions that improved the exposition, and
to Larry Guth for bringing to my attention the relationship of this approach to minimal surfaces. I thank the referees for their constructive comments that made
this paper more readable.

\section {Uryson width of compact metric spaces}   

We prove in this section the generalization of theorem \ref{volume} for metric spaces conjectured by Guth. There are some technicalities in the
proof as we work with Hausdorff content which is not a measure. We explain in the end how
can one give a simpler proof in the manifold case using Hausdorff measure (see remark \ref{measure}).

\begin{Def} The $n$-\textit{dimensional Hausdorff content} $HC_n(U)$ of a subset of a metric space $X$ is the
infimum of $\sum _{i=1}^{\infty }r_i^n$ over all coverings of $U$ by countably many balls $B(x_i,r_i)$.
\end{Def}
We will need a slight variation of Hausdorff content-this will allow us to sidestep the problem that Hausdorff content
is not a measure so it is not additive:

\begin{Def} The \textit{$\zeta $-restricted} $n$-\textit{dimensional Hausdorff content} $HC^{\zeta}_n(U)$ of a subset of a metric space $X$ is the
infimum of $\sum _{i=1}^{\infty }r_i^n$ over all coverings of $U$ by countably many balls $B(x_i,r_i)$ where $r_i\leq \zeta $ for all $i$.
\end{Def}

Clearly we have $HC^{\zeta}_n(U)\geq HC_n(U)$. We remark that if $U$ is contained in a ball of radius $\zeta $ then $HC^{\zeta}_n(U)=HC_n(U)$.

{\bf Notation}. We denote by $B(x,r)$ the open metric ball of radius $r$ and center $x$ and by $\bar B(x,r)$ 
the closed ball. When we don't care about the center we denote it by $B(r)$ ($\bar B(r)$ respectively). We denote
by $S(x,r)$ the sphere of radius $r$ and center $x$, and we denote this by $S_r$ when the center is obvious. Finally we denote by $B(r_2)\setminus B(r_1)$ the annulus between two concentric metric balls. 

The co-area formula \cite[ Theorem 13.4.2]{BZ} will be our main tool. As we will work in the context of metric spaces
it will be crucial below that there is a co-area inequality that applies to Hausdorff content as was shown recently in \cite{LLNR}.
We state this here for $\zeta $-restricted Hausdorff content.

\begin{Lem}\label{co-area} (Lemma 5.3 of \cite{LLNR}) Let $U\subset B(r_2)\setminus B(r_1)$ be a closed set of a proper metric space. Then
$$\int _{r_1}^{r_2}HC^{\zeta}_{n-1}(S_r\cap U)\,dr\leq 2HC^{\zeta}_n(U) $$
where the integral is the Upper Lebesgue integral. The same inequality applies to the Hausdorff content.
\end{Lem}
\begin{proof} We outline the proof of this from \cite{LLNR} for the reader's convenience. If $B(R)$ is a ball 
and $S_r$ is a sphere then $S_r\cap B(R)$ is contained in a ball of radius $\leq R$ for any $r$,
so $HC_{n-1}(S_r\cap B(R))\leq R^{n-1}$ for any $r$. So if $B(R)$ is a ball contained in an annulus $B(r_2)\setminus B(r_1)$ and $\zeta \geq R$
we have $$\int _{r_1}^{r_2}HC_{n-1}^{\zeta}(S_r\cap B(R))\,dr\leq 2R\cdot R^{n-1} \  \ (*).$$
Note now that if $U$ is any closed set for any $\epsilon >0$ there is a covering of $U$ by finitely many balls $B_i(r_i)$, $i=1,...,k$ so that $r_i\leq \zeta $ and
$\sum _{i=1}^k r_i^n-HC^{\zeta}_n(U)<\epsilon $ and the result follows by $(*)$.
Clearly this proof applies to $HC_n(U)$ as well.
\end{proof}

Guth conjectured in \cite{Gu2} that if a compact (or even proper) metric space has locally small $n$-Hausdorff content then it has small Uryson width. We treat now the easier case $n=1$.

\begin{Lem}\label{n=1}  Let $X$ be a proper metric space and let $R>0$. If for any $x\in X$ the $1$-dimensional Hausdorff content of the ball $B(x,R)$ is bounded by
$\dfrac{1}{100} R$ then $UW_{0}(X)\leq R$.
\end{Lem}
\begin{proof} We set $\delta=\dfrac{1}{100} R$. We fix $x_0\in X$ and we consider the closed annuli $A_k=\{x\in X: 10(k-1)R\leq d(x_0,x)\leq 10kR\}$, $k\geq 1, k\in \mathbb N$.
Each $A_k$ is compact so it has a finite covering by balls $B_j(r_j)$ such that $r_j\leq 2\delta $ for all $j$. Let
 $$a_k=HC_1^{2\delta}(A_k).$$ 
We pick for each $A_k$ a covering by open balls $B_j(r_j)$ such that $$\sum r_j-a_k<\delta \  \  \  \  (*) .$$
By doing this for all $k$ we obtain a covering $\mathcal U$ of $X$ by open balls.

%
%
%

Suppose that we have a finite sequence of balls in $\mathcal U$,
$ B_1(r_1),..., B_n(r_n)$ such that $B_i(r_i)$ intersects $ B_{i+1}(r_{i+1})$ for all $i$.
We claim that if this happens then $$\sum _{i=1}^nr_i\leq 10\delta .$$

We may assume by taking a smaller $n$ if necessary and arguing by contradiction that
$$12\delta \geq \sum _{i=1}^nr_i>10\delta .$$

So all these balls are contained in a ball $B(x,R)$ which is contained either in a single annulus $A_k$ or in a union of two annuli $A_k\cup A_{k+1}$. However by our hypothesis the content of $B(x,R)$ is bounded by $\delta $, so we could replace these balls in $\mathcal U$ by finitely many balls $B_s(r_s)$, $s\in S$ such that their union contains $B(x,R)$ and $$\sum _{s\in S}r_s<2\delta .$$

It follows that the sequence $ B_1(r_1),..., B_n(r_n)$
violates $(*)$ for at least one of $A_k,A_{k+1}$.

Let $B\in \mathcal U$.
We note now that if $ B_1(r_1),..., B_n(r_n)$ is a finite sequence of  balls from $\mathcal U$ containing $B$
such that $ B_i(r_i)$ intersects $ B_{i+1}(r_{i+1})$ their union has diameter $<R/2$.

We introduce an equivalence relation on $\mathcal U$. We say that two balls $B,B'$ in $\mathcal U$ are equivalent if there is a finite
sequence of balls $B_1=B,B_2,...,B_n=B'$ such that any two successive balls in the sequence intersect.

We replace then each equivalence class of balls from $\mathcal U$  by their union.

In this way we obtain a cover of $X$ by sets say $D_i$, $i\in \mathbb N$ such that each $D_i$ is open (as a finite union of open balls),  and closed (since its complement is open). It follows that the map
$f:X\to \mathbb N$ where $f(D_k)=k$ is continuous and $$\diam f^{-1}(k)=\diam D_k<R$$
so $UW_{0}(X)\leq R$.

\end{proof}

If $U$ is an open subset of a Riemannian manifold then $\vol _n(U)$ is equal to the $n$-Hausdorff measure of $U$
which is in turn greater or equal to the $n$-dimensional Hausdorff content. It follows that Theorem \ref{volume} is a corollary of
the theorem that we state now-which was conjectured by Guth and proven recently by Liokumovich-Lishak-Nabutovsky-Rotman \cite{LLNR}:

\begin{Thm} \label{content} There is an $\epsilon _n>0$ such that the following holds.
If $X$ is a compact metric space such that for any $x\in X$ the $n$-dimensional Hausdorff content of the ball $B(x,R)$ is bounded by
$\epsilon _n R^n$ then $UW_{n-1}(X)\leq R$.
\end{Thm}
\begin{proof}
We will prove by induction on $n$ that there is a continuous map $\pi: X\to \Sigma$ where $\Sigma$ is a  finite simplicial complex of dimension $\leq n-1$ such that
$\diam \,\pi^{-1}(y)\leq R$ for any $y\in \Sigma$. The theorem holds for $n=1$ by lemma \ref{n=1}.

\begin{Def} Let $Z\subseteq X$ closed. We say that $Z$ is a $D$-{\it separating subset}
if 
$$X\setminus Z=\bigsqcup _{i\in I} U_i$$ where the $U_i$ are open disjoint sets of diameter $\leq D$
and $I$ is finite. We say that the open sets $U_i$ are the {\it pieces of the decomposition of $X$
by $Z$}.

\end{Def}

We set $\zeta =R/1000$.
Let $b(D)$ be the infimum of $HC^{\zeta}_{n-1}(Z)$ over all $D$-separating sets $Z$.
It is not clear whether there exists a $D$-separating set realizing $b(D)$ however it will be sufficient for us to consider sets with content close enough to $b(D)$: We say that $Z$ is a
$\delta ${\it -minimal $D$-separating set} if $Z$ is $D$-separating and
$$HC^{\zeta}_{n-1}(Z)-b(D)\leq \delta .$$
In what follows our statements will be true for $\delta $ sufficiently small.

%

The theorem follows from the next lemma:

\begin{Lem}  \label{induction}  
There is an $\epsilon _n>0$ such that the following holds.
If $X$ is a compact metric space such that for any $x\in X$ the $n$-dimensional Hausdorff content of the ball $B(x,R)$ is bounded by
$\epsilon _n R^n$ then there is a finite simplicial complex $\Sigma $ of dimension $\leq n-1$ and a continuous map $f:X\to \Sigma $  such that:
%
$\diam \, f^{-1}(e)\leq R$ for any simplex $e\in \Sigma $.

\end{Lem}
\begin{proof}  We prove this by induction on $n$. For $n=1$ the statement follows by lemma \ref{n=1}. In particular we may take $\epsilon _1=1/100$.

We will show that the lemma holds for $\epsilon _n$ where we define $\epsilon _n$ inductively by $\epsilon _n= \epsilon _{n-1}/1000^{n+1}$.

We assume now that the lemma holds for $n-1$ for some $n\geq 2$.


\begin{Lem}  \label{small-cut} 
Let $\epsilon _{n-1}$ be the constant provided by lemma \ref{induction} and let $\epsilon _n= \epsilon _{n-1}/1000^{n+1}$.
Let $X$ be a compact metric space  such that for any $x\in X$ the $n$-dimensional Hausdorff content of the ball $B(x,R)$ is bounded by
$\epsilon _n R^n$ .
Let $Z$ be a $\delta $-minimal $R/4$-separating subset of $X$.
Then for any ball of radius $R/1000$,  $B(x,R/1000)$,$$HC^{\zeta}_{n-1}(Z\cap B(x,R/1000))\leq \epsilon _{n-1}\big(\frac{R}{1000}\big)^{n-1}.$$ 

\end{Lem}
\begin{proof}  

We argue by contradiction assuming that $Z$ does not satisfy this inequality for some $x$. We take 

We note that $(R/1000)^n\geq \epsilon _nR^n$. It follows that $HC_{n}(B(x,R))=HC^{\zeta}_{n}(B(x,R))$.
By the co-area inequality (lemma \ref{co-area}) and our hypothesis that 
$HC^{\zeta}_{n}(B(x,R))\leq \epsilon _n R^n$ we have that for some $r\in [R/100,R/50]$ 
$$HC^{\zeta}_{n-1} S(x,r)\leq 200\epsilon _n R^{n-1}\leq \frac {\epsilon _{n-1}R^{n-1}}{5\cdot 1000^{n}} .$$
  
If $Z_1=S(x,r)$ and $Z_2=B(x,r)\cap Z$ we set $Z'=(Z\setminus Z_2)\cup Z_1$. 
We claim that $Z'$ is $R/4$-separating. Indeed let
$$X\setminus Z=\bigsqcup _{i\in I} U_i$$ where $I$ is finite and the $U_i$ are open disjoint sets of diameter $\leq R/4$. Let $U=B(x,r)$. Then
$$X\setminus Z'=\bigsqcup _{i\in I} (U_i\setminus \bar B(x,r))\sqcup U.$$

If $B_i(r_i), i\in I$ is a cover of $Z$ by balls of radius $\leq \zeta$ so that $$\sum _{i\in I}r_i^{n-1} -HC^{\zeta}_{n-1}(Z)<\delta $$
we get a cover of $Z'$ by omitting all balls from this cover intersecting $B(x,R/1000)$ and adding appropriately balls that cover $S(x,r)$ and
approximate $HC^{\zeta}_{n-1} S(x,r)$ up to $\delta$.

We have then
$$HC^{\zeta}_{n-1}(Z')\leq HC^{\zeta}_{n-1}(Z)-\epsilon _{n-1}\big(\frac{R}{1000}\big)^{n-1}+\frac {\epsilon _{n-1}R^{n-1}}{5\cdot 1000^{n}}+\delta$$ 
contradicting the $\delta $-minimality property of $Z$ if we take $$\delta < \frac {\epsilon _{n-1}R^{n-1}}{1000^{n}}.$$

\end{proof}
We prove now lemma \ref{induction}. Let $Z$ be a $\delta $-minimal $R/4$-separating subset of $X$. By lemma \ref{small-cut} and our inductive hypothesis there is a continuous map $\pi _1:Z\to \Sigma_1$ 
where $\Sigma _1$ is a finite  simplicial complex of dimension $\leq n-2$ such that
$\diam\, \pi _1 ^{-1} (e)\leq R/1000$ for any simplex $e\in \Sigma_1$.

Let $U$ be a piece of the decomposition of $X$ by $Z$. Clearly $\bd U\subset Z$ so $\pi _1(\bd U)$ is contained in a finite subcomplex of $\Sigma_1$.
We denote by $\Sigma_{U}$ the minimal such subcomplex of $\Sigma_1$. 

We define a new simplicial complex $\Sigma$ as follows: For each closure of a connected component $U$ we consider the cone $C_{U}$ over $\Sigma_U$
(which is a simplicial complex of dimension $\leq n-1$). We glue $C_U$ to $\Sigma _1$ along
their common subcomplex $\Sigma _U$.

We will need some facts from topology that we recall now (see eg \cite{Hu}). Any finite simplicial complex is an Absolute Neighborhood Retract (ANR). A contractible
ANR is an Absolute Retract (AR). In particular the cone of a finite simplicial complex is an AR. A space $A$ is an AR if and only if it is an
absolute extensor i.e. if it has the following property: if $B$ is any metric space, $K\subseteq B$ is closed and $f:K\to A$ is continuous then $f$ can be extended continuously
to the whole of $B$.

By the above facts it follows that for each $U$ the map $\pi _1:\bd U\to \Sigma _U\subset C_U$
can be extended to a continuous map $\pi :U \to C_U\subset \Sigma $.
Since $X$ is the union of $Z$ with the pieces of the decomposition of $X$ by $Z$ and since 
the map $\pi $ is continuous on the closure of each piece we have that the map $\pi: X\to \Sigma $ is continuous.

Let $e$ be a maximal simplex of $\Sigma $. Then $e$ is either a simplex of $\Sigma _1$ or a cone of
a simplex $e'$ of $\Sigma _1$. If $\pi (U)$ intersects $e$ then in the first case $\bd U$ intersects
$\pi _1^{-1}(e)$ while in the second case $\bd U$ intersects
$\pi _1^{-1}(e')$ . Since $$\diam \,\pi _1^{-1}(e')\leq R/1000\text{ and }\diam (U)\leq R/4$$
we have that $$\diam\, \pi ^{-1}(e) \leq R.$$

We note that an extension argument similar to the one given above appears in section 6.1 of \cite{LLNR}.


\end{proof}

Clearly the theorem follows from the lemma as any point of $\Sigma $ is contained in some simplex $e$ of $\Sigma $.

\end{proof}
\begin{Rem}\label{measure} In the manifold case one could use Hausdorff measure instead of Hausdorff content to
prove theorem \ref{volume}. This would simplify a bit the proof, in particular the proof of Lemma \ref{small-cut}.
We note however that the subset $Z$ that `cuts' the space in small pieces that we introduce is not a manifold. So for the
proof to work one needs a version of the co-area inequality that applies to spaces with finite Hausdorff measure.
We observe that such an inequality follows from Lemma \ref{co-area} by taking a limit as the diameter of the balls approaches $0$ (see also \cite{HS} 1.11, p.15 or \cite{Fe} 3.2.11).
\end{Rem}

\section {The general case and a refinement}

We recall that a metric space is called \textit{proper} if any closed ball is compact.
Theorem \ref{content} holds more generally for proper metric spaces rather than compact ones.
We state here the corresponding inductive statement and explain the modifications needed to prove this.

\begin{Thm}  \label{induction-gen}  
There is an $\epsilon _n>0$ such that the following holds.
If $X$ is a proper metric space such that for any $x\in X$ the $n$-dimensional Hausdorff content of the ball $B(x,R)$ is bounded by
$\epsilon _n R^n$ then there is a locally finite simplicial complex $\Sigma $ of dimension $\leq n-1$ and a continuous map $f:X\to \Sigma $  such that:
%
$\diam \, f^{-1}(e)\leq R$ for any simplex $e\in \Sigma $.
In particular $UW_{n-1}(X)\leq R$.

\end{Thm}
\begin{proof}  The proof is as before by induction on $n$. For $n=1$ the statement follows by lemma \ref{n=1}.
We generalize slightly the definition of  $D$-{\it separating subset}:
\begin{Def} Let $Z\subseteq X$ closed. We say that $Z$ is a $D$-{\it separating subset}
if 
$$X\setminus Z=\bigsqcup _{i\in I} U_i$$ where the $U_i$ are open disjoint sets of diameter $\leq D$
and any ball $B(x,r)$ intersects finitely many of the $U_i$'s.  We say that the open sets $U_i$ are the {\it pieces of the decomposition of $X$
by $Z$}.

\end{Def}

To do the inductive step we fix $x_0\in X$ and let $$A_n=\{x\in X: 10(n-1)R\leq d(x,x_0)\leq 10nR\},$$ $$ B_n=\{x\in X: 10(n-1)R+5R\leq d(x,x_0)\leq 10nR+5R\}$$
($n\geq 1$). We set $\zeta=R/1000$. Each $A_n, B_n$ is compact so we may apply lemma \ref{content}. We modify slightly
lemma \ref{small-cut}: we pick a smaller $\epsilon _n$, say, $$\epsilon _n\leq \frac {\epsilon _{n-1}}{10\cdot 1000^{n+1}}$$
and we obtain the following slightly stronger conclusion by the same proof:
\begin{Lem}  \label{small-cut2} Let  $Z_n$ be a $\delta $-minimal $R/4$-separating set of $A_n$.
Then for any ball of radius $R/1000$,  $B(x,R/1000)$ of $A_n$,$$HC^{\zeta}_{n-1}(Z_n\cap B(x,R/1000))\leq \frac{1}{10}\epsilon _{n-1}\big(\frac{R}{1000}\big)^{n-1}.$$ 
\end{Lem}
The same lemma applies of course for $\delta $-minimal $R/4$-separating sets of $B_n$ which we denote by $T_n$.

Then if $$Z=\bigcup _{n=1}^{\infty}Z_n, \ \ \ T=\bigcup _{n=1}^{\infty}T_n $$ we claim that $Z\cup T$ is an $R/4$-separating set of $X$. Indeed if $$U_n^i, \  \, V_n^j, \, i\in I_n, j\in J_n$$ are the pieces of the decomposition of
$A_n$ by $Z_n$, respectively $B_n$ by $T_n$ we set 
$$I=\cup_{n=1}^{\infty } I_n,\, J=\cup_{n=1}^{\infty } J_n$$ and
$$I'=\{i\in I: U_i\text{ is open in X} \}, \, J'=\{j\in J: V_j\text{ is open in X} \}.$$

Then we may take the open sets $$W_{ij}=U_i\cap V_j, \, i\in I',\, j\in J'$$
to be the pieces of the decomposition of $X$ by $Z\cup T$. Clearly $\bigcup W_{ij}=X\setminus (Z\cup T)$. 
Each ball intersects finitely many of these sets since it intersects finitely many of $U_i,V_j$. Applying lemma \ref{small-cut2} we see that $Z\cup T$ satisfies the hypothesis of
theorem \ref{induction-gen} for $n-1$ and the same proof as in theorem \ref{induction} applies in this case too.

\end{proof}

Our method allows us to give a weaker sufficient condition for a space to have small Uryson width. We remark that the volume condition given in Theorem \ref{content} for a space to have small Uryson width is not a necessary one. For example one can have a ball $B$ of radius $R$ with $HC_n(B)=R^n$ and $UW_1(B)<\epsilon $
for arbitrarily small $\epsilon $.  Recall that a metric space $X$ has inductive dimension $\leq n$ if every point $x\in X$ has arbitrarily small neighborhoods $U$ such that $\dim\,\bd U\leq n-1$ (see \cite{HW}, def. III.1, p.24). It turns out that for separable metric spaces the inductive dimension is equal to the covering (known also as topological) dimension (see \cite{HW}, thm V 8, p.67).

We give a quantitative statement similar to the definition of inductive dimension where we assume for $\bd U$ small $n-1$-Hausdorff content rather than dimension. So this is a result in the spirit of passing from a qualititative theorem
from topological dimension theory into some quantitative estimates (for more on this theme see \cite{Gu2}, sec. 0.2).

\begin{Thm}\label{inductive} There exist $\epsilon _n>0$ so that the following holds. Suppose $X$ is a proper metric space
and there exists
a radius $R$ such that every ball $B(x,R)$ in $X$ is contained in an open set $U$ with the properties:

1. $U\subseteq B(x,10R)$.

2. $HC_{n-1}(\bd U)\leq \epsilon _nR^{n-1}$.

Then $UW_{n-1}(X)\leq R$.
\end{Thm}

\begin{proof} 
One argues as in the proof of theorem \ref{induction-gen}. Here we don't need the co-area formula as we assume that $HC_{n-1}(\bd U)$ is small.
So we can prove as in lemma \ref{small-cut} that a $\delta $-minimal $1000R$ separating subset $Z$ has locally small $HC_{n-1}^{\zeta}$
and the rest of the proof is identical.
\end{proof}

We remark however that this theorem does not provide a characterization of spaces $X$ with small $UW_{n-1}(X)$. For example consider a 3-regular metric tree
$T$ where each edge has length $\delta >0$. Let $X=T\times [0,\epsilon]$. Clearly $UW_1(X)\leq \epsilon $ but by picking $\delta $ very small
we see that $HC_{1}(\bd U)\geq R$ for open sets $U$ containing an $R$-ball. 

\section {An example}

Balacheff-Sabourau showed in \cite{BS} that there is a constant $c_g$ such that if $S_g$ is a Riemannian surface of genus $g$ and volume 1 then there is a map $f:S_g\to \mathbb R$ such that the length of the level sets $f^{-1}(x)$ is bounded by $c_g$. Guth
in \cite {Gu4}, p. 765 asks if there is a similar bound for maps from the 3-sphere to $\mathbb R^2$ where we assume
again that $\vol\,S^3=1$.
In this section we show that for any $C>0$ there is a Riemannian metric $g$ on the 3-sphere with $\vol  (S^3,g)=1$ such that for any map $f:S^3\to \mathbb R^2$ there is some $x\in \mathbb R^2$ for which $\diam (f^{-1}(x))>C$ answering this question in the negative.
We note that Theorem 1.1 implies that there is $C>0$ such that for any Riemannian  metric $g$ $(S^3,g)$ of volume 1 there a map $f:S^3\to \Sigma$
for some 2-complex $\Sigma $ such that for any $x\in \Sigma $  $\diam (f^{-1}(x))\leq C$. So the example really shows that we can not replace $\Sigma $ by $\mathbb R^2$. It remains open whether one can find a map $f:S^3\to \Sigma $ where $\Sigma $ is a 2-complex
so that the length of the level sets (rather than the diameter) is bounded (see \cite{Gu4}).

One could see maps $f:M\to \mathbb R^n$ where $M$ is a manifold as `higher analogs' of Morse
functions-especially if one imposes some smoothness conditions on $f$. Such maps were studied extensively in \cite{Gro-sing1}
and in \cite{Gro-sing2}.

Our construction is based on a simple observation about graphs. Namely the fact that
the non-planarity of the complete graph $K_5$ has a quantitative version: if we take a copy of $K_5$ where the edges are very long
then in any drawing of this graph on the plane two `far away' points get identified. We give a formal proof of this in the next two lemmas.


%
%
%
Let $\Gamma=(V,E)$ be a graph.
We consider here metric graphs, that is edges are isometric to some interval $[0,\ell]$ and the distance between two points is the length of the shortest path joining them.

It will be convenient to identify some planar subsets with graphs: We will say that a subset of the plane $\mathbb R^2$ is a graph if it is a finite union of closed arcs and any two
arcs may intersect only at their endpoints. We say that these arcs are the edges of the graph. We call the subset of the plane consisting of endpoints of these arcs
the set of vertices of the graph. We note that a subset of the plane can have more than one graph structure
as one can always subdivide some edges and add vertices. So when we talk of a graph on the plane we assume that
we have fixed a choice of vertices and edges.

We recall a classical topological lemma (see Corollary 31.6 in \cite{Wil}).

\begin{Lem}\label{simple} Let $X$ be a metric space and let $f:[0,1]\to X$ be a continuous path joining $x,y$. Then there is a simple
path joining $x,y$ contained in $f([0,1])$.
\end{Lem}

%

\begin{Lem}\label{planar} Let $K_5=(V,E)$ be the complete graph with 5 vertices metrized so that each edge has length $10R$.
If
$f:K_5 \to \mathbb R^2$ is continuous then there is some $x\in \mathbb R^2$ with $\diam f^{-1}(x)>R$.

\end{Lem}

\begin{proof} 
We argue by contradiction, that is we assume that $\diam f^{-1}(x)\leq R$ for all $x$.
For any vertex $v\in V$ we consider the ball of radius $2R$, $B(v,2R)$ in $K_5$.
This is a tree $T$ consisting of a union of intervals $[v,x_i],\,i=1,2,3,4$ intersecting only at $v$. 
By lemma \ref{simple} there is a simple arc $\alpha _i$ with endpoint $f(v),f(x_i)$ contained in $f([v,x_i])$.
Let $p_1$ be the last intersection point on $\alpha _2$ between $\alpha _1,\alpha _2$. We replace then $\alpha _2$ by the subarc
 of $\alpha _2$ with endpoints $p_1,f(x_2)$. Call this arc $\beta _2$
Similarly we consider the last intersection point $p_2$ of $\alpha _3$ with $\alpha _1\cup \beta _2$ and we replace $\alpha _3$
by its subarc of $\alpha _3$ with endpoints $p_2,f(v_3)$. We repeat this with the other arcs and
we note that $\alpha _1\cup \beta _2\cup ...\cup \beta _4$ is a tree
with 4 (or 5) endpoints contained in $f(T)$. We do this for all 5 vertices of $K_5$ and we obtain 5 trees on the plane
$T_i,i=1,...,5$. By lemma \ref{simple} for each pair $T_i,T_j$ there is a simple arc $\alpha_{ij}$ joining an endpoint of $T_i$
with an endpoint of $T_j$. We note that $\alpha _{ij}$ might intersect $T_i,T_j$ at many points. We replace $\alpha _{ij}$
by a subarc $\beta _{ij}$ such that $\beta _{ij}$ intersects $T_i\cup T_j$ only at its endpoints and one endpoint is in $T_i$
and another is in $T_j$. By our hypothesis the arcs $\beta_{ij}$ do not intersect each other and the trees $T_i$ do not intersect
either. So the union of the $T_i$'s with the $\beta_{ij}$'s is a planar graph. However it is clear that this graph
has $K_5$ as a minor-just contract the trees $T_i$ to points. This is clearly a contradiction since $K_5$ is not planar.

\end{proof}

\textbf{The construction.}  The idea is that we may construct a 3-sphere with volume 1 that 
contains an almost isometric copy of $K_5$ where the edges of $K_5$ are very long. The point is we can replace
$K_5$ with a very thin handlebody $\Sigma $ and $S^3$ can be obtained by gluing two such thin handlebodies. To arrange that
$K_5$ is almost isometrically embedded we interpolate $\Sigma \times [0,L]$ for some big $L$ between the two copies.
We explain now this more formally.

Given $C>0$ pick a graph $\Gamma $ isomorphic to $K_5$ so that the edges of $\Gamma $ have length
$100C$. Thicken $\Gamma$ so that is becomes a handlebody $\Sigma $. We may choose a Riemannian metric on $\Sigma $ so that
$\Sigma $ is contained in an $\epsilon $ neighborhood of $\Gamma $, its volume is less than $\epsilon $, its boundary
surface $S$ and $\Sigma $ are both $(1,\epsilon )$ quasi-isometric to $\Gamma $ and $\area (S)<\epsilon $, where $\epsilon $ is a small positive number that we will specify later.
Glue to $\Sigma $ along $S$ the manifold $S\times [0,100C]$ with the product metric to obtain a handlebody $M$. Clearly
 $\vol (M)<(100C+1)\epsilon $.
Consider now a small copy $S_1$ of $S$ of diameter $<1/10$ embedded in the standard sphere $S^3$ of volume $1-\delta $ where $\delta $
is a small positive number to specify later. Denote the handlebody of diameter $<1/10$ bounded by $S_1$ in $S^3$ by $\Sigma _1$.
Consider now a Riemannian metric $g_1$ on $S\times [0,1]$ interpolating between the metric on $\bd M$ and $S_1$, so
$S\times \{0\}$ is isometric to $S$ and $S\times \{1\}$ is isometric to $S_1$. We may further assume that the volume of
$(S\times [0,1],g_1)$ is smaller than $2\epsilon $. We glue now $S\times \{0\}$ to $M$ and $\bd M$ and $S\times \{1\}$ to $S^3\setminus S_1$ by isometries. By appropriately picking $\delta, \epsilon $ and smoothing the metric we obtain a Riemannian sphere
$(S^3,g)$ with $\vol (S^3,g)=1$. Since $\Gamma $ is $(1,\epsilon)$ quasi-isometrically embedded in this sphere clearly
lemma \ref{planar} implies that for any $f:S^3\to \mathbb R^2$ there are $x,y\in \Gamma $ with $d(x,y)>C$ and $f(x)=f(y)$.

We note that this construction is very similar to a construction in proposition 0.10 in \cite{Gu2} which is turn is similar
to example ($H_1''$) of \cite{Gro-width}.

Larry Guth after seeing this example suggested another way to obtain such metrics: start with a tripod $T$ with long edges. Clearly any
map from the tripod to $\mathbb R$ has some fiber of long diameter. Consider $X=T\times [0,L]$ for some big $L$. Again any map
$f:X\to \mathbb R^2$ has some big fiber. Now thicken $X$ to obtain a 3-ball and glue two such balls along their boundaries to obtain
a 3-sphere $S$ such that for any $f:S\to \mathbb R^2$ there is a point $x\in \mathbb R^2$ for which $\diam f^{-1}(x)$ is large.

\end{document}
\bye